\documentclass[12pt,a4paper]{amsart}
\usepackage{amsfonts}
\usepackage{epsfig}
\usepackage{graphicx}
\usepackage{amsmath}
\usepackage{amssymb}
\usepackage{amsthm}
\usepackage{amscd}

\newcounter{main}

\newtheorem{theorem}{Theorem}[section]
\newtheorem{proposition}[theorem]{Proposition}
\newtheorem{lemma}[theorem]{Lemma}
\newtheorem{corollary}[theorem]{Corollary}

\newtheorem{definition}{Definition}[section]

\newcommand{\blanksquare}{\,\,\,$\sqcup\!\!\!\!\sqcap$}

\newcounter{example}
{{\stepcounter{example}}{\flushleft {\bf Example \arabic{example}:}}}%
{\par}

\title[The spectrum of infinite dimensional random products]{On the spectrum of infinite dimensional random products of compact operators}

\author{ M\'{a}rio Bessa and Maria Carvalho }

\date{September 21, 2007}

\begin{document}
\maketitle

\begin{abstract}
We consider an infinite dimensional separable Hilbert space and its
family of compact integrable cocycles over a dynamical system $f$.
Assuming that $f$ acts in a compact Hausdorff space $X$ and
preserves a Borel regular ergodic measure which is positive on
non-empty open sets, we conclude that there is a residual subset of
cocycles within which, for almost every $x$, either the
Oseledets-Ruelle's decomposition along the orbit of $x$ is dominated
or has a trivial spectrum.
\end{abstract}

\bigskip

\emph{keywords:} Random operators; dominated splitting; multiplicative ergodic theorem; Lyapunov exponents.\\

\begin{section}{Introduction}
Let $\mathcal{H}$ be an infinite dimensional separable Hilbert space
and $\mathcal{C}(\mathcal{H})$ the set of linear compact operators
acting in $\mathcal{H}$ with the uniform norm given by
$$\| \mathcal{T}\|=\underset{v\neq 0}{\sup} \frac{\| \mathcal{T}(v)\|}{\|v\|}.$$
Consider a homeomorphism $f:X\rightarrow{X}$ of a compact Hausdorff
space $X$ and $\mu$ an $f$-invariant Borel regular measure that is
positive on non-empty open subsets. Given a family $(A(x))_{x \in
X}$ of operators in $\mathcal{C}(\mathcal{H})$ and a continuous
vector bundle $\pi: X \times \mathcal{H} \rightarrow {X}$, we define
the \emph{associated cocycle over} $f$ by
$$
\begin{array}{cccc}
F(A): & X\times{\mathcal{H}} & \longrightarrow & X\times{\mathcal{H}} \\
& (x,v) & \longmapsto & (f(x),A(x)\cdot v).
\end{array}
$$
The map $F$ satisfies the equality $\pi\circ{F}=f\circ{\pi}$ and,
for all $x\in X$, $F_{x}(A):\mathcal{H}\rightarrow{\mathcal{H}}$ is
linear on the fiber $\mathcal{H}:=\pi^{-1}(\{x\})$. For simplicity
of notation we call $A$ a \textit{cocycle}.\\

A \textit{random product} of a cocycle
$A:X\rightarrow\mathcal{C}(\mathcal{H})$ \textit{associated to the
map $f$} is the sequence, indexed by $x\in X$, of linear maps of
$\mathcal{H}$ defined, for each $n \in{\mathbb{N}_{0}}$, by
$A^{0}(x)=Id$ and
$$A^{n}(x)=A(f^{n-1}(x))\circ...\circ A(f(x))\circ A(x).$$ In this paper
we are interested in the asymptotic properties of
random products, that is, the limit of the spectra, as $n$ goes to
$\infty$, of the sequence $(A^{n}(x))_{n \in \mathbb{N}}$, for most
points $x$. In general it is not guaranteed, not even in a relevant
subset of $X$, the convergence of the sequence of operators
$(A^{n}(x))_{n \in \mathbb{N}}$ or of their spectra. But under the
hypothesis that $A$ is integrable, that is,
$$\int_{X}\log^{+}\|A(x)\| \, d\mu(x)<\infty,$$ where
$\log^{+}(y)=\text{max}\,\{0,\log(y)\}$, the theorem of Ruelle
(\cite{Ru}) offers, for $\mu$-almost every point $x\in X$, a nice
description of a complete set of Lyapunov exponents and associated
$A$-invariant directions. The aim of this work is to identify
generic properties of these exponents and corresponding
decomposition.

The approach in Mate's work (\cite{M}), where it is assumed that $A$
is a bounded operator, $f$ is the shift of $N$ symbols and, for
every $x$, the sequence $(A^{n}(x))_{n \in \mathbb{N}}$ converges,
suggests that the null cocycle has a main role in this context: we
may split the Hilbert space into a direct sum of two subspaces, one
that aggregates all the fixed directions and the other corresponding
to the eigenvalue zero (that is, the Lyapunov exponent $-\infty$).
Among compact cocycles this scenario should be improved. In fact,
for these operators the unique point of accumulation of the spectrum
is $0$ and therefore the component of the spectrum that may lie on
the unit circle (inducing non-hyperbolicity) is finite dimensional;
besides if the spectrum is trivial (reduced to one point), then the
compact operator has to be the null one. Nevertheless the success of
Mate's result, which does not depend on perturbations, is strongly
based upon the hypothesis that, for every shift orbit, the sequence
of operators $(A^{n}(x))_{n \in \mathbb{N}}$ converges. Without this
assumption, the best we can expect is an approximate result stating
that, generically, either the above Mate's decomposition reduces to
the null part or is, in some sense, hyperbolic.

The main difficulty, due to the infinite dimensional environment, is
precisely to cancel the spectrum by a small perturbation of the
original system. In the context of families of finite dimensional
linear invertible cocycles, Bochi and Viana (\cite{BV2}) managed to
prove that, by a $C^0$-small perturbation, we may reach a cocycle
exhibiting, for almost every point, uniform hyperbolicity in a
finite projective space or else a one-point spectrum Oseledets'
decomposition. By \textit{hyperbolicity} the authors mean the
existence, for $\mu$-almost every $x\in X$, of an $A(x)$-invariant
decomposition of the fiber $\mathcal{H}$ into a direct sum of two
invariant subspaces $E^{1}_{x}\oplus{E^{2}_{x}}$ which varies
continuously with the point $x$ and enhances a stronger contraction,
or a weaker expansion, by $A$ along the first one. In our setting we
could apply directly this result to a $C^{0}$-approximation of $A$
with finite rank (see \cite{RS}); however this straight application
would endorse a meagre result: it only gives a $C^{0}$-dense
panorama, instead of the aimed $C^{0}$-residual one; besides, in the
case a dominated splitting prevails, this would be a decomposition
of just a finite dimensional subspace of $\mathcal{H}$.

Essentially all we need is a strategy to perturb and therefore to
produce a residual dichotomy; this has to be done without leaving
the world $C_{I}^{0}(X,\mathcal{C}(\mathcal{H}))$ of continuous
compact integrable cocycles and keeping control on the possibly
infinite amount of Lyapunov exponents, most of which may be equal to
$-\infty$. Two key ingredients in the argument of \cite{BV2} can be
adapted to our infinite dimensional context: the upper
semi-continuity of a map that measures how the sum of Lyapunov
exponents behaves; and the extension and continuity of a do\-mi\-na\-ted
splitting. Concerning the first one, we had to accept that now this
map has infinite components and, due to the presence of the Lyapunov
exponent $-\infty$, may take values on the extended real set, which
may prevent integrability. This difficulty is the reason for
assuming that $\mu$ is ergodic and positive on non-empty open sets.
But this is the main difference, in the large its intervention is
the same as in \cite{BV2}. The second one is harder to deal with
because $\mathcal{H}$ is infinite dimensional and $A$, being
compact, is not invertible - and it may even happen that $\inf \, \{
\|A(x)\|: x \in X \}=0$. The notion of dominated splitting must then
be reformulated and applied to Oseledets-Ruelle's splittings where
the stronger space is associated to the first $k \in \mathbb{N}$
finite Lyapunov exponents (whose sum of multiplicities gives the
\emph{index} of the splitting) and the weaker subspace corresponds
to the remaining ones: this way the first subspace is finite
dimensional and there the restriction of $A$ is invertible. Let us
see how we proceed from here.

As $C_{I}^{0}(X,\mathcal{C}(\mathcal{H}))$ is a Baire space and each
$p^{th}$-component of the entropy map is upper semi-continuous (see
Section 2.5), each has a re\-si\-du\-al subset $\mathcal{R}_{p}$ of
continuity points. The set $\cap \mathcal{R}_{p}$ is also residual
and its elements are points of continuity of all of these
map-components. We take one of them, say $A$, and apply to it
Ruelle's theorem (see Section 2.2). As $\mu$ is ergodic, the
Lyapunov exponents of $A(x)$ and co\-rres\-pon\-ding multiplicities are
constant for $\mu$-almost every $x$. Besides, as $\mu$ is positive
on non-empty open subsets, the properties that are valid $\mu$ -
almost everywhere are also dense.

By compactness of the operators, if the Lyapunov exponents of $A$
are all equal, then they must be $-\infty$ and so the limit operator
given by Ruelle's theorem is identically null. Assume now that the
Lyapunov exponents of $A$ are not all equal. The space $X$ can then
be sliced into measurable strata within each of which the
Oseledets-Ruelle's decomposition induces a direct sum
$\mathcal{H}=E_{1} \oplus E_{2}$ where the dimension of $E_{1}$ is
constant, $E_{1}$ is associated to some finite number of the first
finite Lyapunov exponents, and the splitting is dominated. If the
union of these slices has full measure, the proof is complete.
Otherwise, we can find a subset with positive measure where neither
the Oseledets-Ruelle's splitting is dominated nor the Lyapunov
exponents are all equal. This allows us to diminish drastically, by
a small global perturbation, the value of one of the components of
the entropy map, contradicting its continuity at $A$. Accordingly we
establish that:

\begin{theorem}\label{teorema1}
There exists a $C^{0}$-residual subset $\mathcal{R}$ of the set of
integrable compact cocycles
${C_{I}^{0}(X,\mathcal{C}(\mathcal{H}))}$ such that, for
$A\in{\mathcal{R}}$ and $\mu$-almost every $x\in{X}$, either the
limit
$\underset{n\rightarrow{\infty}}{\text{lim}}({A(x)^{*}}^{n}A(x)^{n})^{\frac{1}{2n}}$
is the null operator or the Oseledets-Ruelle's splitting of $A$
along the orbit of $x$ is dominated.
\end{theorem}
\end{section}

\begin{section}{Preliminary results}

\begin{subsection}{Completeness}
\begin{lemma}\label{Baire property}
$C_{I}^{0}(X,\mathcal{C}(\mathcal{H}))$ is a Baire space.
\end{lemma}
\begin{proof}
Since $\mathcal{H}$ is complete, the space
$\mathcal{C}(\mathcal{H})$, with the uniform norm, is also complete
(see \cite{RS}). The space $C^{0}(X,\mathcal{C}(\mathcal{H}))$ is
endowed with the norm defined by
$$\|A\|=\underset{x \in X} {\max} \, \|A(x)\|$$ and this way it is
complete: if $(A_{n})_{n \in \mathbb{N}}$ is a Cauchy sequence in \linebreak
$C^{0}(X,\mathcal{C}(\mathcal{H}))$ then, for each $x \in X$, the
sequence $(A_{n}(x))_{n \in \mathbb{N}}$ has the Cauchy property and
therefore converges in $\mathcal{C}(\mathcal{H})$. This defines a
limit of $(A_{n})_{n \in \mathbb{N}}$ in
$C^{0}(X,\mathcal{C}(\mathcal{H}))$.

Consider now a Cauchy sequence, say $(B_{n})_{n \in \mathbb{N}}$, of
elements of the subspace $C_{I}^{0}(X,\mathcal{C}(\mathcal{H}))$,
that is, continuous compact cocycles such that, for all $n \in
\mathbb{N}$, $\int_{X}\log^{+}\|B_{n}(x)\| \,d\mu(x)<\infty$. Then:
\begin{itemize}
\item $(B_{n})_{n \in \mathbb{N}}$ converges to some $B \in
C^{0}(X,\mathcal{C}(\mathcal{H}))$.
\item As $X$ is compact and $B$ is continuous, there exists $M > 0$
such that, for all $x \in X$, we have $\|B(x)\| \leq M$; and so $0
\leq \log^{+}\|B(x)\| \leq \log(M)$.
\item As $(\|B_{n}\|)_{n}$ converge uniformly to
$\|B\|$, the same holds for the sequence of $\mu$ - integrable maps
$(\log^{+}\|B_{n}\|)_{n}$, and therefore $\log^{+}(\|B\|)$ is $\mu$
- integrable.
\item Besides
$0\leq \int_{X}\log^{+}\|B(x)\| \,d\mu(x)\leq\log(M)<\infty$.
\end{itemize} \end{proof}
\end{subsection}

\begin{subsection}{The multiplicative ergodic theorem}
The following result gives a spectral decomposition for the limit of
random products of compact cocycles under the previously defined
integrability condition.

\begin{theorem}\label{Ruelle}(Ruelle~\cite{R})
Let $f:X\rightarrow X$ be a homeomorphism and $\mu$ any
$f$-invariant Borel probability. If $A$ belongs to
$C^{0}_{I}(X,\mathcal{C}(\mathcal{H}))$, then, for $\mu$-a.e
$x\in{X}$, we have the following properties:
\begin{enumerate}
\item[(a)] The limit
$\underset{n\rightarrow{\infty}}{\text{lim}}({A(x)^{*}}^{n}A(x)^{n})^{\frac{1}{2n}}$
exists and is a compact operator $\mathcal{L}(x)$, where $A^{*}$
denotes the dual operator of $A$.
\item[(b)] Let
$e^{\lambda_{1}(x)}>e^{\lambda_{2}(x)}>...$ be the nonzero
eigenvalues of $\mathcal{L}(x)$ and $U_{1}(x)$, $U_{2}(x)$, ... the
associated eigenspaces whose dimensions are denoted by $n_{i}(x)$.
The sequence of real functions $\lambda_{i}(x)$, called
\textbf{Lyapunov exponents} of $A$, where $1\leq i(x) \leq j(x)$ and
$j(x)\in \mathbb{N} \cup \{\infty\}$ verifies:
\begin{enumerate}
\item[(b.1)] The functions $\lambda_{i}(x)$, $i(x)$, $j(x)$ and $n_{i}(x)$
are $f$-invariant and depend in a measurable way on $x$.
\item[(b.2)] Let $V_{i}(x)$ be the orthogonal complement of
$U_{1}(x)\oplus{U_{2}(x)}\oplus{...}\oplus{U_{i-1}(x)}$ for
$i<j(x)+1$ and $V_{j(x)+1}(x)=Ker(\mathcal{L}(x))$. Then:
\begin{enumerate}
\item[(i)] $\underset{n\rightarrow{\infty}}{\text{lim}}\frac{1}{n}\log\|A^{n}(x)u\|=\lambda_{i}(x) \text{ if }u\in{V_{i}(x)\setminus V_{i+1}(x)}\\\text{ and } i<j(x)+1$;
\item[(ii)] $\underset{n\rightarrow{\infty}}{\text{lim}}\frac{1}{n}\log\|A^{n}(x)u\|=-\infty\text{ if }u\in{V_{j(x)+1}(x)}$.
\end{enumerate}
\end{enumerate}
\end{enumerate}
\end{theorem}

Notice that, as $\mu$ is ergodic, the maps $i(x)$, $j(x)$,
$n_{i}(x)$ and $\lambda_{i}(x)$ are constant $\mu$-almost
everywhere. Besides, as $\mathcal{L}(x)$ is a compact operator, if
its eigenvalues are all equal, then they must be all zero, that is,
the Lyapunov exponents of $A$ at $x$ are all equal to $-\infty$.

In the sequel we will denote by $\mathcal{O}(A)$ the full measure
set of points given by this theorem. Since $\mu$ is positive on
non-empty open subsets, $\mathcal{O}(A)$ is dense in $X$.
\end{subsection}

\begin{subsection}{Dimension}
The infinite dimension of $\mathcal{H}$ brings additional trouble
while dealing with Oseledets' decompositions because in the sequel
we will need one of them with finite codimension. This is the aim of
next lemma.

\begin{lemma}\label{finite dimension}
Let $A$ be an integrable compact operator and $\lambda_{i}(x)$,
$U_{i}(x)$ as in Ruelle's theorem. If $\lambda_{i}(x)\neq -\infty$,
then $U_{i}(x)$ has finite dimension.
\end{lemma}
\begin{proof}
The numbers $e^{\lambda_{1}(x)}\ > \ e^{\lambda_{2}(x)}\ >...$,
where $\lambda_{k}(x)$ is different from $-\infty$, are the nonzero
eigenvalues of the compact operator $\mathcal{L}(x)$, and
$U_{1}(x)$, $U_{2}(x)$, ... the associated eigenspaces. By
compactness of $\mathcal{L}(x)$, theses spaces have finite
dimensions (see \cite{RS}).
\end{proof}
\end{subsection}

\begin{subsection}{Dominated splittings}
Given $f$ and $A$ as above and an $f$-in\-va\-ri\-ant set $\mathcal{K}$,
we say that a splitting $E_{1}(x)\oplus E_{2}(x)=\mathcal{H}$ is
\textit{$\ell$-dominated} in $\mathcal{K}$ if $A(E_{i}(x)) \subset
E_{i} (f(x))$ for every $x\in \mathcal{K}$, the dimension of
$E_{i}(x)$ is constant in $\mathcal{K}$ for $i=1,2$, and there are
$\theta_{\mathcal{K}}> 0$ and $\ell \in \mathbb{N}$ such that, for
every $x \in \mathcal{K}$ and any pair of unit vectors $u\in
E_{2}(x)$ and $v\in E_{1}(x)$, one has
$$\|A(x)(v)\| \geq \theta_{\mathcal{K}}$$
$$\frac{\|A^{\ell}(x)u\|}{\|A^{\ell}(x)v\|} \leq \frac{1}{2}.$$ This definition corresponds
to hyperbolicity in an infinite dimensional projective space; we will denote it by $E_{1} \succ_{\ell} E_{2}$. \\

The splittings we are interested in are the ones corresponding to
Lyapunov subspaces given by Ruelle's theorem. In this setting:
\begin{definition}
Given an $f$-invariant set $\mathcal{K}$ contained in
$\mathcal{O}(A)$, the Oseledets-Ruelle's decomposition is
$\ell$-dominated in $\mathcal{K}$ if we may detach in it a direct
sum of two subspaces, say $E_{1}(x)\oplus E_{2}(x)=\mathcal{H}$,
such that $E_{1}(x)$ is associated to a finite number of the first
Lyapunov exponents, say $\lambda_{1}, \, \lambda_{2}, \, ... \,
\lambda_{k}$, the subspace $E_{2}(x)$ corresponds to the remaining
ones and $E_{1} \succ_{\ell} E_{2}$.
\end{definition}

The classical concept of domination in the finite dimensional
setting is stronger than this one, requiring a comparison of the
strength of each Oseledets' subspace with the next one. Due to the
possible presence of $-\infty$ in the set of Lyapunov exponents,
this is in general unattainable in our context, unless this exponent
does not turn up.

Besides, for future use of the $\ell$-domination, we require that
the norm of $A$ in $\mathcal{K}$ is bounded away from zero. In fact,
among finite dimensional automorphisms, \textit{domination} implies
that the angle between any two subbundles of the dominated splitting
is uniformly bounded away from zero, a very useful property while
proving that the dominated splitting extends continuously. Due to
the lack of compactness of $\mathcal{O}(A)$ and the fact that we are
dealing with a family $(A(x))_{x}$ of compact operators acting on an
infinite dimensional space - so $A(x)$ is not invertible and its
norm may not be uniformly bounded away from zero - we cannot expect
such a strong statement in our setting, unless we relate, as we have
done in the definition, domination with non-zero norms.

The statement of next lemma ensures that we may check if $x$ in
$\mathcal{O}(A)$ has a dominated Oseledets-Ruelle's decomposition
$\mathcal{H}=E_{1}(x)\oplus E_{2}(x)$, where $E_{1}(x)$ is the
Lyapunov subspace associated to the first $k$ finite Lyapunov
exponents $\lambda_{1} > \lambda_{2}> ... > \lambda_{k}> -\infty$
and $E_{2}(x)$ corresponds to the remaining ones. In what follows we
will address always to this specific Oseledets-Ruelle's splitting.

\begin{lemma}\label{A in E-1 is invertible}
Let $A$ be an integrable compact cocycle acting on an infinite
dimensional Hilbert space $\mathcal{H}$. Consider $x$ in
$\mathcal{O}(A)$, $\lambda_{1}\ > \lambda_{2}> ... > \lambda_{k}$
the first $k$ Lyapunov exponents and $E_{1}(x)=U_{1}(x)\oplus
U_{2}(x)\oplus...\oplus U_{k}(x)$ the corresponding subspace. If
$\lambda_{k}> -\infty$, then the restriction of the operator $A(x):
E_{1}(x) \rightarrow E_{1}(f(x))$ is invertible and $A^{-1}(f(x))$
is compact.
\end{lemma}

\begin{proof}
Let us first check that this restriction of $A(x)$ is injective.
Consider $v\neq 0$ in $E_{1}(x)$. Then, by Ruelle's theorem, one has
$$\underset{n\rightarrow{\infty}}{\text{lim}}\frac{1}{n}\log\|A^{n}(x)v \| \geq
\lambda_{k}> -\infty,$$ so $A(x)v$ cannot be zero. Now, by Lemma 2.3
and since $\mu$ is ergodic, the dimension of $E_{1}(x)$ is finite
and constant in $\mathcal{O}(A)$. Therefore the map $A(x): E_{1}(x)
\rightarrow E_{1}(f(x))$ is an injective linear function between
spaces of equal finite dimension, and so it is surjective. The
compactness of the inverse of $A$, given at each point by a finite
dimensional matrix, now follows.\end{proof}

We need now to verify that domination is easily inherited by
neighbors.

\begin{proposition}\label{extension to the closure of the domination}
If the Oseledets-Ruelle's splitting $E_{1}(x) \oplus
E_{2}(x)=\mathcal{H}$ is $\ell$-dominated over an invariant set
$\mathcal{K} \subset \mathcal{O}(A)$, it may be extended
continuously to an $\ell$-dominated splitting over the closure of
$\mathcal{K}$.
\end{proposition}

\begin{proof} To extend $E_{1}$ we will take advantage from the fact that,
for each $z \in \mathcal{K}$, the subspace $E_{1}(z)$ is an
Oseledets-Ruelle's space associated to a finite number of the first
finite Lyapunov exponents; and, moreover, that $A$ is compact. The
definition of $E_{2}$ at the closure is suggested by our need to
extend the relation $ \succ_{\ell}$, which is easier if the vectors
at the extension are just accumulation points of sequences of
vectors from where the $\ell$-domination holds.

\begin{lemma}\label{angle}
The angle between $E_{1}(z)$ and $E_{2}(z)$, where $z$ belongs to
$\mathcal{K}$, is uniformly bounded away from zero (say bigger than
a constant $\gamma \in \, ]0,\frac{\pi}{2}]$).
\end{lemma}

\begin{proof}
Assume that there are sequences ${(x_{n})}_{n \in \mathbb{N}}$ in
$\mathcal{K}$, ${(u_{n})}_{n \in \mathbb{N}}$ in $E_{2}(x_{n})$ and
${(v_{n})}_{n \in \mathbb{N}}$ in $E_{1}(x_{n})$ such that, for all
$n$, we have $\|u_{n}\|=\|v_{n}\|=1$ and $u_{n} - v_{n}$ converges
to $0$. As $A$ is continuous, if $n$ is large enough then the norm
$\|A^{\ell}(x_{n})(u_{n}) - A^{\ell}(x_{n})(v_{n})\|$ is arbitrarily
small; moreover
$$\|A^{\ell}(x_{n})(u_{n}) - A^{\ell}(x_{n})(v_{n})\| \geq \|A^{\ell}(x_{n})(u_{n})\| - \|A^{\ell}(x_{n})(v_{n})\|$$ and the last difference is equal to
$$\|A^{\ell}(x_{n})(v_{n})\|\left(\frac{\|A^{\ell}(x_{n})(u_{n})\|}{\|A^{\ell}(x_{n})(v_{n})\|}
- 1\right).$$ As there is $\theta_{\mathcal{K}}$, independent of
$x_{n}$ and $v_{n}$, such that $\|A(x_{n})(v_{n})\| \geq
\theta_{\mathcal{K}}$, we have $\|A^{\ell}(x_{n})(v_{n})\| \geq
\theta_{\mathcal{K}}^{\ell}$ and so
$$\frac{\|A^{\ell}(x_{n})(u_{n})\|}{\|A^{\ell}(x_{n})(v_{n})\|} - 1 \approx 0. $$ But
this contradicts the fact that $E_{1}(x_{n}) \succ_{\ell}
E_{2}(x_{n})$ for all $n$.
\end{proof}

Finally consider a sequence ${(x_{n})}_{n \in \mathbb{N}}$ of
elements of $\mathcal{K}$ converging to $x \in X$ and suppose that
along ${(x_{n})}_{n \in \mathbb{N}}$ we may find an $\ell$-dominated
splitting $E_{1,n}\oplus E_{2,n}$ made of Oseledets-Ruelle's
subspaces as mentioned. Recall that $E_{1,n}$ has dimension $p$ for
all $n$ and corresponds to the Lyapunov exponents $\lambda_{1} >
\lambda_{2}
> ... > \lambda_{k} > -\infty$ and that we must extend this dominated splitting to $x$.

Take, for each $n$, a unitary basis $v_{1,n}, ..., v_{p,n}$ of
$E_{1,n}$. As proved above, in Lemma 2.4, since $\lambda_{k}>
-\infty$, the restriction of the operator $A(x): E_{1}(x)
\rightarrow E_{1}(f(x))$ is invertible and $A^{-1}(f(x))$ is
compact. Therefore, for each $i=1,...,p$, the sequence
$(A(x_{n})^{-1}v_{i,n})_{n \in \mathbb{N}}$ has a subsequence
convergent to $h_{i} \in \mathcal{H}$. Apply now the operator $A$ to
obtain $p$ vectors
$w_{i}$ in the fiber at $x$.\\

\textit{Claim: Each $w_{i} \neq 0$}\\

In fact, by continuity of the operator $A$, $w_{i}$ is the limit of
a sub\-sequen\-ce of $(v_{i,n})_{n \in \mathbb{N}}$ and these vectors
satisfy the condition
$$\underset{m\rightarrow{\infty}}{\lim}\frac{1}{m}\log \|A^{m}(x_{n})v_{i,n}\|
\geq \lambda_{k}> -\infty$$ which implies that, if $m$ is big
enough, $\|A^{m}(x)w_{i}\| \geq \exp(m \times \lambda_{k})$; this
prevents $w_{i}$ from being the vector zero.\\

We now define $E_{1}(x)$ as the space spanned by ${w_{1}, ...,
w_{p}}$ and $E_{2}(x)$ as the set of accumulation points, when $n$
goes to $\infty$, of all sequences of vectors in $E_{2,n}(x_{n})$,
where $(x_{n})_{n \in \mathbb{N}}$ is any sequence converging to
$x$. Notice that, as $A$ is continuous and the spaces
$E_{i,n}(x_{n})$ are determined by Lyapunov exponents, we have
$A(E_{i}(x)) \subset E_{i}(f(x))$.\\

The vectors $w_{i}$ could be linearly dependent, so the dimension of
$E_{1}(x)$ might be less than $p$. However, as verified above, the
angle between $E_{1}(z)$ and $E_{2}(z)$, where $z$ belongs to
$\mathcal{K}$, is uniformly bounded away from zero; this prevents
directions of $E_{1}(z)$ from mixing with those of $E_{2}(z)$, the
way $E_{1}$ had to loose dimension. Therefore the dimension of
$E_{1}$ is constant and equal to $p$ in the closure of
$\mathcal{K}$. And, as we will verify, forms with $E_{2}$ a direct
sum in $\mathcal{H}$ which is a dominated splitting.

\begin{lemma}\label{subspace and direct sum}\item[]
\begin{enumerate}
\item[(a)] $E_{2}(x)$ is a subspace of $\mathcal{H}$.
\item[(b)] $\mathcal{H}=E_{1}(x) \oplus E_{2}(x)$.
\end{enumerate}
\end{lemma}
\begin{proof}
(a) The vector zero is in $E_{2}(x)$ as limit of the null sequence
of vectors of the subspaces $E_{2,n}(x_{n})$. Consider now a scalar
$\eta$ and two non-zero vectors $u_{0}$ and $u_{1}$ of $E_{2}(x)$.
By definition of $E_{2}(x)$, there are sequences $(u_{0,n})_{n \in
\mathbb{N}}$ and $(u_{1,n})_{n \in \mathbb{N}}$ of
$(E_{2,n}(x_{n}))_{n \in \mathbb{N}}$ such that $u_{0}$ is the limit
of the former and $u_{1}$ is the limit of the latter. Then, for each
$n$, the sum $u_{0,n}+u_{1,n}$ and the product $\eta \, u_{0,n}$ are
in the subspace $E_{2,n}(x_{n})$ and converge to $u_{0}+u_{1}$ and
$\eta \, u_{0}$ respectively.\\

\noindent(b) For each $n \in \mathbb{N}$, we have
$\mathcal{H}=E_{1,n}(x_{n}) \oplus E_{2,n}(x_{n})$; therefore, given
$h \in \mathcal{H}$, there are vectors $e_{1,n} \in E_{1,n}(x_{n})$
and $e_{2,n} \in E_{2,n}(x_{n})$ such that $h=e_{1,n}+e_{2,n}$. As
verified in the previous Lemma, the sequence $(e_{1,n})_{n \in
\mathbb{N}}$ has a convergent subsequence to a vector $e \in
E_{1}(x)$. Then the corresponding subsequence of $(e_{2,n})_{n \in
\mathbb{N}}$ converges to $h-e$, which accordingly belongs to
$E_{2}(x)$. Then $h=e+(h-e)$ is in $E_{1}(x)+E_{2}(x)$.

Moreover, if $E_{1}(x) \cap E_{2}(x) \neq \{0\}$, then there is a
vector $u \neq 0$ which is the limit of a sequence $(u_{1,n})_{n \in
\mathbb{N}}$ of vectors in $E_{1,n}(x_{n})$ and also the limit of a
sequence $(u_{2,n})_{n \in \mathbb{N}}$ inside $E_{2,n}(x_{n})$. But
this implies that the angle between these subspaces must be
arbitrarily close to zero as $n$ goes to $\infty$, which contradicts
what was proved above. Therefore $\mathcal{H}=E_{1}(x) \oplus
E_{2}(x)$.
\end{proof}

\begin{lemma}\label{domination}
$E_{1}(x) \succ_{\ell} E_{2}(x)$.
\end{lemma}
\begin{proof}
We must check that the $\ell$-domination of the splitting at $x_{n}$
is inherited by this choice of spaces at $x$. Fix $u \in E_{2}(x)$
and $v \in E_{1}(x)\setminus \{0\}$.

We know that $v$ is a linear combination of ${w_{1}, ..., w_{p}}$
and so it is the limit of a sequence of vectors of $E_{1,n}$, say
$(v_{n})_{n \in \mathbb{N}}$. Therefore $A^{\ell}(x)v \neq 0$
because, for $m$ big enough, the iterates $A^{m}(x)v$ inherit at
least the minimum rate of growing of $w_{i}$, that is, $\exp
(\lambda_{k})$. Besides, as $\|A(x_{n})v_{n}\| \geq
\theta_{\mathcal{K}}$ for all $n$, the same inequality holds in the
limit as $n$ goes to $\infty$, that is, $\|A(x)v\| \geq
\theta_{\mathcal{K}}$.

By definition of $E_{2}(x)$, there is a sequence in
$E_{2,n}(x_{n})$, say $(u_{n})_{n \in \mathbb{N}}$, that converges
to $u$. Since we have, for all $n$, $E_{1,n}(x_{n}) \succ_{\ell}
E_{2,n}(x_{n})$, taking limits on the inequality
$$\frac{\|A^{\ell}(x_{n})u_{n}\|}{\|A^{\ell}(x_{n})v_{n}\|} \leq
\frac{1}{2},$$ we get $$\frac{\|A^{\ell}(x)u\|}{\|A^{\ell}(x)v\|}
\leq \frac{1}{2}.$$ \\
We emphasize that, if $x \notin \mathcal{K}$, we do not known
whether $E_{1}(x)$ and $E_{2}(x)$ are Oseledets-Ruelle's subspaces,
since we cannot guarantee that $x \in \mathcal{O}(A)$.
\end{proof}

\begin{corollary}\label{continuity of the subbundles}
The subbundle $E_{i}(x)$, for $i=1,2$, is well defined and continuous.
\end{corollary}
\begin{proof}
By definition, $E_{2}(x)$ does not depend on the sequence
$(x_{n})_{n \in \mathbb{N}}$ of elements of $\mathcal{K}$ converging
to $x$. Concerning $E_{1}(x)$:

(i) If $x \notin \mathcal{K}$, the $\ell$-domination ensures that
$E_{1}(x)$ is unique since its dimension is fixed: the iterates of
its unit vectors grow faster than those of any unit vector in
$E_{2}(x)$.

(ii) If $x \in \mathcal{K}$, these two subspaces coincide with the
Oseledets-Ruelle's spaces already assigned to $x$. In fact, while
constructing $E_{1}(x)$, we must consider the constant sequence
equal to $x$; the Oseledets-Ruelle's subspace at $x$ corresponding
to the Lyapunov exponents\linebreak $\lambda_{1}, \, \lambda_{2}, \, ...
\,\lambda_{k}$ is then contained in $E_{1}(x)$ and has the same
dimension, so these two spaces coincide.

Uniqueness implies that these spaces vary continuously.
\end{proof}

This ends the proof of Proposition 2.5.
\end{proof}
\end{subsection}

\begin{subsection}{Upper semi-continuity of the entropy function}\label{four}
Given a bounded linear operator
$A:\mathcal{H}\rightarrow{\mathcal{H}}$ and a positive integer $p$,
let $\wedge^{p}(\mathcal{H})$ be the $p^{th}$ exterior power of
$\mathcal{H}$, that is, the infinite dimensional space generated by
$p$ vectors of the form $e_{1}\wedge{e_{2}\wedge...\wedge{e_{p}}}$
with $e_{i} \in \mathcal{H}$. The operator $A$ induces another one
on this space defined by
$$\wedge^{p}(A)(e_{1}\wedge{e_{2}\wedge...\wedge{e_{p}}})=A(e_{1})\wedge{A(e_{2})\wedge...\wedge{A(e_{p})}}.$$
The space $\wedge^{p}(\mathcal{H})$ has endowed the inner product
such that\linebreak $\|e_{1}\wedge{e_{2}\wedge...\wedge{e_{p}}}\|$ is the
$p^{th}$-dimensional volume of the parallelepiped spanned by $e_{1},
e_{2}, ..., e_{p}$. The cocycle $\wedge^{p}(A)$ is continuous with
respect to the associated norm. For details see \cite{T}, chapter V.

\begin{lemma}\label{compactness and integrability}
If $A$ is compact and integrable, then $\wedge^{p}(A)$ also is.
\end{lemma}
\begin{proof}
Fix $p$ and $A$ and take any bounded sequence $(y_{n})_{n}$ of
elements of $\wedge^{p}\mathcal{H}$; we must prove that there exists
a subsequence $(y_{n_{k}})_{k}$ such that
$(\wedge^{p}(A)(y_{n_{k}}))_{k}$ converges. For each
$n\in\mathbb{N}$, let
$y_{n}={v_{1}}^{n}\wedge{...}\wedge{v_{p}}^{n}$, with
${v_{j}}^{n}\in\mathcal{H}$ for all $j=1,...,p$. Hence
$\wedge^{p}(A)(y_{n})=A({v_{1}}^{n})\wedge{...}\wedge
A({v_{p}^{n}})$. As $(y_{n})_{n}$ is bounded, for each $j=1,...,p$
the sequence $({v_{i}}^{n})_{n}$ is bounded in $\mathcal{H}$.
Therefore, since $A$ is compact, for all $j=1,...,p$ the sequence
$A({v_{j}^{n}})_{n}$ admits a subsequence convergent to
$u_{j}\in{\mathcal{H}}$. That is, there are subsets of $\mathbb{N}$,
say $\mathbb{N}_{1}$, $\mathbb{N}_{2}$, ..., $\mathbb{N}_{p}$ such
that
\begin{enumerate}
\item $\mathbb{N}_{j} \supseteq \mathbb{N}_{j+1}$ for all $j$
\item $A({v_{j}^{n}})_{n\in\mathbb{N}_{j}}$ converges to
$u_{j}$.
\end{enumerate}
Therefore the subsequence $(y_{n})_{n\in\mathbb{N}_{p}}$ converges
to $u_{1}\wedge{...}\wedge u_{p}$.

According to the definition of the inner product in
$\wedge^{p}(\mathcal{H})$, for all $x$ we have
$$\|\wedge^{p}(A)(x)\| \leq \|A(x)\|^{p}$$ and so, as $A$ is
integrable,
\begin{eqnarray*}
\int_{X}\log^{+}\|\wedge^{p}(A)(x)\| \,d\mu(x) &\leq&
\int_{X}\log^{+}\|A(x)\|^{p} \,d\mu(x) \\
&=& p \, \int_{X}\log^{+}\|A(x)\| \,d\mu(x) < \infty.
\end{eqnarray*}
\end{proof}

Since $\wedge^{p}(A)$ is compact and integrable we can apply to it
Theorem~\ref{Ruelle} and conclude that, for $\mu$-a.e. $x$,
$$\underset{n\rightarrow{+\infty}}{\text{lim}}\frac{1}{n}\log \| {(\wedge^{p}A)}^{n}(x)\|=\lambda_{1}^{\wedge^{p}}(x).$$
This is the largest Lyapunov exponent given by the dynamics of the
operator $\wedge^{p}(A)$ at $x$. Moreover, for $\mu$-a.e. $x$, we
have
$$\lambda_{1}^{\wedge^{p}}(x)=\sum_{i=1}^{p}\lambda_{i}(x)$$
and
$$\lambda_{1}^{\wedge^{p}}(x)=\underset{n\rightarrow{+\infty}}{\text{lim}}\frac{1}{n}\log
\|\wedge^{p}(A^{n}(x))\|$$ (see~\cite{A}). In fact, as $\mu$ is
ergodic, this equality reduces, $\mu$-a.e. $x$, to
$$\lambda_{1}^{\wedge^{p}}=\sum_{i=1}^{p}\lambda_{i}=\underset{n\rightarrow{+\infty}}{\text{lim}}\frac{1}{n}\log
\|\wedge^{p}(A^{n}(x))\|.$$

Given $p\in{\mathbb{N}}$ define the $p^{th}$\emph{-entropy function}
by
$$
\begin{array}{cccc}
LE_{p}: & C_{I}^{0}(X,\mathcal{C}(\mathcal{H})) & \longrightarrow & \mathbb{R}\cup \{-\infty\} \\
& A &\longmapsto& \sum_{i=1}^{p}\lambda_{i}(A).
\end{array}
$$
As the Lyapunov exponents vary in a measurable way, there is no
reason to expect the function $LE_{p}$ to be continuous. However, as
$\mu$ is ergodic and positive on non-empty open sets, this function
is upper semi-continuous. Let us see why.

\begin{proposition}\label{upper semi-continuity}
Consider a cocycle $A$ and the sequence given, for each $n \in
\mathbb{N}$, by $a_{n}=\log \|\wedge^{p}(A^{n})\|$. Then
\begin{enumerate}
\item[(i)] $\lambda_{1}^{\wedge^{p}}=\underset{n\rightarrow{+\infty}}{\text{lim}}\frac{a_{n}}{n}$.
\item[(2i)] $(a_{n})_{n \, \in \, \mathbb{N}}$ is sub-additive.
\item[(3i)] $\underset{n\rightarrow{+\infty}}{\text{lim}} \, \frac{a_{n}}{n}=\underset{n \in \mathbb{N}}
{\text{inf}} \, \frac{a_{n}}{n}.$
\item[(4i)] For each $n \in \mathbb{N}$, the map $A \longrightarrow a_{n}(A)$ is continuous.
\end{enumerate}
\end{proposition}

\begin{proof}
Concerning (i):\\

\emph{Case 1: $\lambda_{1}^{\wedge^{p}} \geq 0$}\\

We know that, for $\mu$ - a.e. $x$,
$$\lambda_{1}^{\wedge^{p}}=\underset{n\rightarrow{+\infty}}{\text{lim}}\frac{1}{n}\log
\|\wedge^{p}(A^{n}(x))\|,$$ so, for each such a $x$ and $n$ big
enough, the norm $\|\wedge^{p}(A^{n}(x))\|$ is approximately
$e^{\lambda_{1}^{\wedge^{p}}n}$, and therefore does not vanish.
Besides, by definition,
$$\frac{1}{n}\log \|\wedge^{p}(A^{n})\|= \frac{1}{n}\log \|\wedge^{p}(A^{n}(t_{n}))\|$$ for a suitable choice
of $t_{n} \in X$. As $\mu$ is positive on non-empty open sets, the
subset $\mathcal{O}(\wedge^{p}(A))$ is dense in $X$ and so, for each
$n$, there exists $z_{n} \in \mathcal{O}(\wedge^{p}(A))$ such that
the distance between $t_{n}$ and $z_{n}$ is sufficiently small in
order to guarantee that
$$\|\wedge^{p}(A^{n}(t_{n})) - \wedge^{p}(A^{n}(z_{n}))\| \approx
0.$$ Besides, as for all $n$ we have

$$\underset{k\rightarrow{+\infty}}{\text{lim}}\frac{1}{k}\log \|\wedge^{p}(A^{k})(z_{n})\|
=\lambda_{1}^{\wedge^{p}}$$ by a diagonal argument we deduce that

$$\underset{n\rightarrow{+\infty}}{\text{lim}}\frac{1}{n}\log \|\wedge^{p}(A^{n})(z_{n})\|=
\lambda_{1}^{\wedge^{p}}.$$ And so, if $n$ is big enough, $\|\wedge^{p}(A^{n}(z_{n}))\| \geq
 e^{\lambda_{1}^{\wedge^{p}}} \geq 1$, and therefore
$$\frac{\|\wedge^{p}(A^{n}(t_{n}))\|}{\|\wedge^{p}(A^{n}(z_{n}))\|} \approx
1.$$ Then
$$\underset{n\rightarrow{+\infty}}{\text{lim}}\frac{1}{n}\log \|\wedge^{p}(A^{n})(t_{n})\|=
\underset{n\rightarrow{+\infty}}{\text{lim}}\frac{1}{n}\log
\frac{\|\wedge^{p}(A^{n}(t_{n}))\|}{\|\wedge^{p}(A^{n}(z_{n}))\|} \,
. \, \|\wedge^{p}(A^{n}(z_{n}))\|$$ which reduces to
$$\underset{n\rightarrow{+\infty}}{\text{lim}}\frac{1}{n}\log \|\wedge^{p}(A^{n}(t_{n}))\|=
\underset{n\rightarrow{+\infty}}{\text{lim}}\frac{1}{n}\log
\|\wedge^{p}(A^{n}(z_{n}))\|= \lambda_{1}^{\wedge^{p}},$$ and means
that
$$\lambda_{1}^{\wedge^{p}}=\underset{n\rightarrow{+\infty}}{\text{lim}} \, \frac{a_{n}}{n}.$$\\

\emph{Case 2: $-\infty < \lambda_{1}^{\wedge^{p}} < 0$}\\

As in the previous case, for all $n$ and a suitable choice of
$t_{n}$, we have
$$\frac{1}{n}\log \|\wedge^{p}(A^{n})\|= \frac{1}{n}\log \|\wedge^{p}(A^{n}(t_{n}))\|;$$ moreover, for each $n$,
there exists $z_{n} \in \mathcal{O}(\wedge^{p}(A))$ such that the
distance between $t_{n}$ and $z_{n}$ is sufficiently small in order
to guarantee that
$$\|\wedge^{p}(A^{n}(t_{n})) - \wedge^{p}(A^{n}(z_{n}))\| \leq e^{\lambda_{1}^{\wedge^{p}}n}.$$
Besides, as $\lambda_{1}^{\wedge^{p}} < 0$, if $n$ is big enough,
$$\|\wedge^{p}(A^{n})(z_{n})\| < 1.$$
Therefore
$$\frac{1}{n}\log \|\wedge^{p}(A^{n}(t_{n}))\| \leq \frac{1}{n}\log
\left(\|\wedge^{p}(A^{n})(z_{n})\| +
e^{\lambda_{1}^{\wedge^{p}}n}\right) \leq \frac{1}{n}\log \left(1 +
e^{\lambda_{1}^{\wedge^{p}}n} \right)$$ and so
$$\underset{n\rightarrow{+\infty}}{\text{lim}}\frac{1}{n}\log
\|\wedge^{p}(A^{n})(t_{n})\|= \lambda_{1}^{\wedge^{p}}.$$ \\

\emph{Case 3: $\lambda_{1}^{\wedge^{p}}= -\infty$}\\

Given $\epsilon > 0$, consider, as above, sequences $(t_{n})_{n \in
\mathbb{N}}$ and $(z_{n})_{n \in \mathbb{N}}$ such that
$$\frac{1}{n}\log \|\wedge^{p}(A^{n})\|=\frac{1}{n}\log \|\wedge^{p}(A^{n}(t_{n}))\|$$
and, if $n$ is big enough,
$$\|\wedge^{p}(A^{n}(t_{n})) - \wedge^{p}(A^{n}(z_{n}))\| \leq e^{-\epsilon
n}$$ and
$$\|\wedge^{p}(A^{n}(z_{n}))\| \leq e^{-\epsilon n}.$$ Then
$$\frac{1}{n}\log \|\wedge^{p}(A^{n}(t_{n}))\| \leq \frac{1}{n}\log \left(2 e^{-\epsilon n}\right)$$
and therefore, for all $\epsilon > 0$,
$$\underset{n\rightarrow{+\infty}}{\text{lim}}\frac{1}{n}\log
\|\wedge^{p}(A^{n})(t_{n})\| \leq -\epsilon$$ which implies that
$$\underset{n\rightarrow{+\infty}}{\text{lim}}\frac{1}{n}\log
\|\wedge^{p}(A^{n})(t_{n})\| = -\infty.$$ \\

(2i) This assertion is the same as the one in Section 2.1.3 of
(see~\cite{BV2}) since the extra value $-\infty$ the sequence may
take does not bring any additional difficulty.\\

(3i) This is a direct consequence of (2i).\\

(4i) For each fixed $n$, the continuity of the map $A \longmapsto
a_{n}(A)$ is ensured by the continuity, with the operator, of the
norm $\|\wedge^{p}(A^{n})\|$.\end{proof}

\begin{corollary}
For all $p$ the function $LE_{p}$ is upper semi-continuous.
\end{corollary}

\begin{proof}
$LE_{p}$ is the infimum of a sequence of continuous functions with
values on the extended real line, and so it is upper semi-continuous
(see \cite{R}).
\end{proof}

\end{subsection}
\end{section}

\begin{section}{Perturbation lemmas}\label{five}
Let us see how, using the absence of domination, to perform
appropriate $C^{0}$-perturbations of our original system to increase
the number of contractive directions.

\begin{lemma}\label{rotation}
Let $A\in C_{I}^{0}(X,\mathcal{C}(\mathcal{H}))$, $x\in X$ and
$\epsilon>0$. For any 2-dimensional subspace $E\subset \mathcal{H}$,
we may find $\xi_{0}>0$ (not depending on $x$) such that, for all
$\xi \in \, ]0,\xi_{0}[$ there exists a measurable integrable
cocycle $B_{\xi}$ such that,
\begin{enumerate}
\item[(a)] $B_{\xi}(x)\cdot u=A(x)\cdot u$, $\forall u\in{E^{\perp}}$;
\item[(b)] $B_{\xi}(x)\cdot u=A(x)\cdot R_{\xi}\cdot u$, $\forall u\in{E}$, where $R_{\xi}$ is the rotation of angle $\xi$ in $E$;
\item[(c)] $\|A-B_{\xi}\|\leq{\epsilon}$.
\end{enumerate}
\end{lemma}
\begin{proof}
If $A=0$, choose $B_{\xi}=A$. Otherwise, consider the direct sum
$\mathcal{H}=E^{\perp}\oplus E$ and denote by
$$R_{\theta}=\begin{pmatrix}
cos(\theta) & -sin(\theta)\\
sin(\theta) & cos(\theta)\\
\end{pmatrix}$$
the matrix of the rotation of angle $\theta$ in an orthonormal basis
of $E$. Then take $\xi_{0}>0$ such that
$\|Id-R_{\xi_{0}}\|\leq{\frac{\epsilon}{\|A\|}}$ and, for each
$v=v_{1}+v_{2}$, where $v_{1} \in E^{\perp}$ and $v_{2} \in E$,
define the perturbation cocycle by

$$B_{\xi}(y)\cdot v= \left\{\begin{array}{ccc}
A(y)\cdot v \,\,\,\,\,\,\,\,\,\,\,\,\,\,\,\,\,\,\,\,\,\,\,\,\,\,\,\,\,\,\,\,\,\,\,\,\,\,\,\,\,\,\,\,\,\,\,\,\text{if } y\not={x} \\
A(x)\cdot v_{1}+A(x)\cdot{R_{\xi}\cdot v_{2}} \,\,\,\,\text{if }
y={x}
\end{array}\right\}.$$ Clearly this cocycle verifies the properties (a), (b) and (c) and
the lemma is proved.
\end{proof}

The following result tells how to interchange directions. The main
idea, coming from Proposition 7.1 of ~\cite{BV2}, is to use the
absence of hyperbolic behavior to concatenate several small
rotations of the form given by Lemma~\ref{rotation}.

\begin{proposition}\label{mixing directions}
Consider a cocycle $A$, $\delta>0$ and $x\in X$ a non-periodic
point endowed with a splitting $\mathcal{H}=E_{1}(x)\oplus E_{2}(x)$ such
that the restriction of $A(x)$ to $E_{1}(x)$ is invertible and, for
some $m(\delta, A)=m\in\mathbb{N}$ large enough, we have
$$\frac{\|A^{m}(x)\vert_{E_{2}}\|}{\|(A^{-1}\vert_{E_{1}})^{m}(x)\|}\geq{1/2}.$$
Then, for each $j=0,...,m-1$, there exists an integrable compact
operator
$$L_{j}:\mathcal{H}\rightarrow\mathcal{H},$$
with $\|L_{j}-A(f^{j}(x))\|<\delta$ and such that $L_{m-1}\circ
...\circ L_{0}(v)=w$ for some nonzero vectors $v\in{E_{1}}$ and $w\in
A^{m}(x)(E_{2})$.
\end{proposition}

We now want to apply this strategy of perturbation to the set of
points where domination fails.

\begin{definition} Let $\Lambda_{p}(A,m)$ be the set of points $x$ such that,
along the orbit of $x$ we have an Oseledets-Ruelle's decomposition of
index $p$ which is $m$-dominated. Denote by
$\Gamma_{p}(A,m)=X\setminus \Lambda_{p}(A,m)$ and by
$\Gamma_{p}^{*}(A,m)$ the set of points in
$\mathcal{O}(A)\cap\Gamma_{p}(A,m)$ which are non-periodic and
verify $\lambda_{p}>\lambda_{p+1}$.
\end{definition}

As checked previously, the set $\Lambda_{p}(A,m)$ is closed and so
all the just mentioned subsets are measurable. Notice also that if
$x$ belongs to $\Gamma_{p}^{*}(A,m)$ for some $m$, then the
$m$-domination on $\mathcal{K}=\{\text {orbit of} \, \,x\}$ of the
Oseledets-Ruelle's splitting may fail by two (possibly coexisting)
events:

\begin{itemize}
\item[(NB)] The norm of the operator $A$ restricted to $E_{1}$ takes values arbitrarily
small along the orbit of $x$.

That is, for all $\theta > 0$ there are $N=N_{\theta, x} \in
\mathbb{N}$ and a unit vector $v_{N} \in E_{1}(f^{N}(x))$ such that
$$\|A(f^{N}(x))(v_{N})\| < \theta.$$ We call ${\Gamma_{p,1}^{*}}$ the set of
points $x \in \Gamma_{p}^{*}(A,m)$ where this happens.

\item[(ND)] The dynamics on the subspace $E_{1}$ does not $m$-dominate the one on $E_{2}$.

This means that there are $n \in \mathbb{N}$ and unit vectors $v_{n}
\in E_{1}(f^{n}(x))$ and $u_{n}\in E_{2}(f^{n}(x))$ such that
$$\frac{\|A^{m}(f^{n}(x))u_{n}\|}{\|A^{m}(f^{n}(x))v_{n}\|} \geq \frac{1}{2}.$$

The points $x \in \Gamma_{p}^{*}(A,m)$ where property (ND) is valid
but not (NB) will be denoted by ${\Gamma_{p,2}^{*}}$.
\end{itemize}

We proceed explaining how we can perform locally and globally a
blending of specific directions of the Oseledets-Ruelle's splitting
for points inside ${\Gamma_{p,1}^{*}} \cup {\Gamma_{p,2}^{*}}$.
Given a point $x$ in ${\Gamma_{p,1}^{*}} \cup {\Gamma_{p,2}^{*}}$,
the aim of next lemmas is to perturb locally $A$, along the orbit of
$x$, in order to carry out an abrupt decay of the norm of the
$p^{th}$-exterior power product $\wedge^{p}$.

\begin{subsection}{Perturbation (p, NB)}

Consider a point $x$ in ${\Gamma_{p,1}^{*}}$. To reduce the norm of
$\wedge^{p}$ we will just replace $A$ by the null operator at the
$n^{th}$-iterate of $x$ along the direction inside $E_{1}(x)$
restricted to which the norm of $A$ is very close to zero.

\begin{lemma}\label{local perturbation 1} Given $\epsilon>0$, there exists
a measurable function \linebreak
$\mathcal{N}:\Gamma_{p,1}^{*}\rightarrow{\mathbb{N}}$ such that, for
$\mu$-almost every $x\in\Gamma_{p,1}^{*}$, every $n\geq
\mathcal{N}(x)$ and each $j=0,...,n$, there exists an integrable
compact operator $L_{j}:\mathcal{H}\rightarrow\mathcal{H}$
satisfying
$$\|L_{j}-A(f^{j}(x))\|<\epsilon$$ and
$$\|\wedge^{p}(L_{n-1}\circ ...\circ L_{0})\|=0.$$
\end{lemma}

\begin{proof} We may assume that $\mu(\Gamma_{p,1}^{*})> 0$, otherwise there is
nothing to prove. Therefore $\mu
\,(\mathcal{O}(A)\cap\Gamma_{p,1}^{*})=\mu \,(\Gamma_{p,1}^{*})$. If
$x \in \mathcal{O}(A)\cap\Gamma_{p,1}^{*}$ we can take $N_{\epsilon,
x}$ as in [NB], define $\mathcal{N}(x)=N_{\epsilon,x}+1$ and choose
a unit vector $v_{\mathcal{N}} \in E_{1}(f^{\mathcal{N}}(x))$ such that
$$\|A(f^{\mathcal{N}}(x))(v_{\mathcal{N}})\| < \epsilon.$$ Then,
for $n \geq \mathcal{N}$, consider
$$L_{i}=A(f^{i}(x)) \text{           for            } i=0,..., \mathcal{N}-1, \mathcal{N}+1, ...,
n$$ and
$$L_{\mathcal{N}}(v)=
\left\{\begin{array}{ll} A(f^{\mathcal{N}}(x))(v) \, \, &
\mbox{if $v \in \, <v_{\mathcal{N}}>^{\perp}$} \\
0 \, \, & \mbox{if $v$ \text{is colinear with}\, $v_{\mathcal{N}}$}
\end{array}
\right.$$ Since $x$ is not periodic by $f$, this family of operators
is well defined. Besides $\|\wedge^{p}(L_{n-1}\circ ...\circ
L_{0})\|=0$. Notice also that, if $\epsilon$ is small, the above
perturbation of $A$ is also minute.
\end{proof}
\end{subsection}

\begin{subsection}{Perturbation (p, ND)}

Consider now a point $x$ in ${\Gamma_{p,2}^{*}}$. If among the $p+1$
first Lyapunov exponents of $A$ the value $-\infty$ is not present,
we can use the argument in~\cite{BV2} to alter the norm of
$\wedge^{p}$. In the case $\lambda_{p+1}=-\infty$ we may take
advantage of the fact that, in the subbundle $E$ associated to the
Lyapunov exponents $\lambda_{j}$ for $j\geq p+1$, the norm
$\|A^{n}(x)|_{E}\|$ is close to zero for $n$ large enough.

\begin{lemma}\label{local perturbation 2}
Consider $\epsilon,\delta>0$. If $m\in\mathbb{N}$ is large enough, then
there exists a measurable function
$\mathcal{N}:\Gamma_{p,2}^{*}\rightarrow{\mathbb{N}}$ such that, for
$\mu$-almost every $x\in\Gamma_{p,2}^{*}$ and every $n\geq
\mathcal{N}(x)$, we may find integrable compact operators 
$L_{0},...,L_{n-1}$ with $\|L_{j}-A(f^{j}(x))\|<\delta$ for each $j$ and satisfying
\begin{enumerate}
\item[(a)] $\|\wedge^{p}(L_{n-1}\circ ...\circ L_{0})\|\leq
e^{n(\lambda_{1}+...+\lambda_{p-1}+\frac{\lambda_{p}+\lambda_{p+1}}{2}+\epsilon)}$
if $\lambda_{p+1}\not=-\infty$;
\item[(b)] $\|\wedge^{p}(L_{n-1}\circ ...\circ L_{0})\|\leq{e^{- \, n\epsilon}}$
\text{    if    }$\lambda_{p+1}=-\infty$.
\end{enumerate}
\end{lemma}

\begin{proof}
The proof of (a) is a direct consequence of Proposition~\ref{mixing
directions} following the argument of Proposition 4.2 of~\cite{BV2}.

Let us now explain the scheme to prove (b). Consider
$\Gamma_{p,2}^{*}$, subset of $\Gamma_{p}^{*}(A,m)$ where $m$ is
large enough as demanded in Proposition~\ref{mixing directions}. We
may assume that $\mu(\Gamma_{p,2}^{*})> 0$, otherwise there is
nothing to prove. By definition of $\Gamma_{p}^{*}(A,m)$, we have
$\lambda_{p}\not=-\infty$. Thus, by Lemma 3.12 of~\cite{Bo}, we
conclude that, for $\mu$-a.e. $x\in\Gamma_{p,2}^{*}$, there exists
$\mathcal{N}_{1}(x)$ such that, for all $n\geq \mathcal{N}_{1}(x)$
and $s\approx n/2$, the iterate $y=f^{s}(x)$ verifies
$$\frac{\|A^{m}(y)\vert_{E_{2}}\|}{\|(A^{-1}\vert_{E_{1}})^{m}(y)\|}\geq{1/2}.$$
We can then apply Proposition~\ref{mixing directions} to such a
generic $x \in \Gamma_{p,2}^{*}$ because, by hypothesis, $x$ is not
periodic and $\mathcal{H}=E_{1,x}\oplus E_{2,x}$, where $E_{1,x}$
has dimension $p$ and corresponds to the finite vector space spanned
by the Lyapunov exponents $\lambda_{1}$,..., \,$\lambda_{p}$ (which
may be not all distinct but whose multiplicities add up to $p$) and
$E_{2,x}$ is associated to the infinite dimensional vector space
spanned by the remaining ones. Therefore, we consider, for
$i=0,...,s,s+m,...,n$, the operators $L_{i}=A(f^{i}(x))$ and, for
the iterates $f^{i}(y)$ with $i=0,...,m-1$, we take $L_{i}$ as given
by Proposition~\ref{mixing directions}.

We need now to evaluate the norm of $\wedge^{p}(L_{n-1}\circ
...\circ L_{0})$. Take $U_{x}$ the subspace associated to the
largest Lyapunov exponent of the $p^{th}$-exterior product, say
$\lambda_{1}^{\wedge p}$, which is given by $\lambda_{1}^{\wedge
p}=\sum_{i=1}^{p}\lambda_{i}$. The space $U_{x}$ is $1$-dimensional
because $\lambda_{p} > \lambda_{p+1}$. Denote by $S_{x}$ the vector
space related to the remaining Lyapunov exponents, which sum up to
$\lambda_{2}^{\wedge p}$. To the splitting
$\wedge^{p}(\mathcal{H})=U\oplus S$ we may apply Lemma 4.4
of~\cite{BV2} and Proposition~\ref{mixing directions} to deduce that
$$\wedge^{p}(L_{m-1}\circ ...\circ L_{0})(y):\wedge^{p}(\mathcal{H}_{y})\rightarrow\wedge^{p}(\mathcal{H}_{f^{m}(y)})$$
satisfies
\begin{equation}\label{mix}
\wedge^{p}(L_{m-1}\circ ...\circ L_{0})(y)(U_{y})\subset
S_{f^{m}(y)}.
\end{equation} If $A_{1}$ denotes the action of $\wedge^{p}(A)$ between $x$ and $y$
and $A_{2}$ denotes the action of $\wedge^{p}(A)$ between
$f^{s+m}(y)$ and $f^{n}(x)$, we can consider a suitable
(Oseledets-Ruelle's) basis with respect to which $A_{1}$, $A_{2}$
and $B:=L_{m-1}\circ ...\circ L_{0}(y)$ are written as the simple
$4$-block ``matrices''

$$A_{1}=\begin{pmatrix}
  A_{1}^{uu} & 0 \\
0 & A_{1}^{ss}
 \end{pmatrix},
\, B=\begin{pmatrix}
  B^{uu} & B^{us} \\
B^{su} & B^{ss}
 \end{pmatrix}
\text{              and             }A_{2}=\begin{pmatrix}
  A_{2}^{uu} & 0 \\
0 & A_{2}^{ss}
\end{pmatrix}$$
where, for $i=1,2$, $A^{uu}_{i}\in\mathbb{R}$ and $A^{ss}_{i}$ is an
infinite dimensional operator. It follows from (\ref{mix}) that
$B^{uu}=0$ and so
\begin{equation}\label{product}
\wedge^{p}(L_{n-1}\circ...\circ L_{0})=\begin{pmatrix}
  0 & A_{2}^{uu}B^{us}A_{1}^{ss} \\
A_{2}^{ss}B^{su}A_{1}^{uu} & A_{2}^{ss}B^{ss}A_{1}^{ss}
 \end{pmatrix}.
\end{equation}
Since $\lambda_{p+1}=-\infty$, we have $\lambda_{2}^{\wedge
p}=-\infty$ and so the operator $A_{i}^{ss}$ ($i=1,2$) is
arbitrarily close to the null one for large choices of $n$.
Moreover, all entries $A_{1}^{uu}$, $A_{2}^{uu}$, $B^{us}$, $B^{su}$
and $B^{ss}$ are bounded. Then it suffices to consider a large $n$,
bigger than $\mathcal{N}_{1}(x)$ and $m$, to reach inequality (b).
\end{proof}
\end{subsection}

Doing these perturbations at $\mu$-a.e. point of
$\Gamma_{p}^{*}(A,m)$ we deduce that

\begin{corollary}\label{global perturbation}
Let $A$ be a cocycle in $C_{I}^{0}(X,\mathcal{C}(\mathcal{H}))$,
$\epsilon>0$ and $\delta>0$. Then there exist $m\in\mathbb{N}$,
$p\in\mathbb{N}$ and a cocycle $B \in
C_{I}^{0}(X,\mathcal{C}(\mathcal{H}))$ with
$\|B-A\|_{\infty}<\delta$, equal to $A$ outside the open set
$\Gamma_{p}(A,m)$ and verifying
$$\lambda_{1}^{\wedge p}(B) \, <
\left\{\begin{array}{ll} [\lambda_{1}^{\wedge _{p-1}}(A) +
\frac{\lambda_{p}(A)+\lambda_{p+1}(A)}{2}] \, + \epsilon &
\mbox{if $\lambda_{p+1}(A) \not= -\infty$} \\
- \, \epsilon & \mbox{if $\lambda_{p+1}(A)= -\infty$}
\end{array}.
\right.$$
\end{corollary}

\begin{proof}
After the previous lemmas, we may follow the argument in Proposition
7.3 of~\cite{BV2}.
\end{proof}
\end{section}

\begin{section}{Proof of Theorem~\ref{teorema1}}

Consider the $p^{th}$\emph{-entropy function} defined by
$$
\begin{array}{cccc}
LE_{p}: & C_{I}^{0}(X,\mathcal{C}(\mathcal{H})) & \longrightarrow & \mathbb{R}\cup \{-\infty\} \\
& A &\longmapsto& \sum_{i=1}^{p}\lambda_{i}(A)
\end{array}
$$
where $(\lambda_{i}(A))_{i=1,...,\infty}$ are the Lyapunov exponents
of the operator $A(x)$, for every $x$ in $\mathcal{O}(A)$. This map
is upper semi-continuous, so it has a residual subset of points of
continuity in the Baire set $C_{I}^{0}(X,\mathcal{C}(\mathcal{H}))$.
Take $A$ in this generic subset, consider $x$ in the
Oseledets-Ruelle's domain $\mathcal{O}(A)$ and denote by
$\mathcal{K}$ the orbit of $x$.

If the Lyapunov exponents of $A(x)$
are all equal, then the proof is complete. Otherwise, if $p \in
\mathbb{N}$ is such that $\lambda_p
> \lambda_{p+1}$, we pursue as follows:

\begin{enumerate}
\item If $\lambda_{p+1} > -\infty$ and $x$ is periodic by $f$, say of period $R$, then along the orbit of $x$
the Oseledets-Ruelle's splitting given by $E_1 \oplus E_2$, where
$E_1$ is the subspace associated to the Lyapunov exponents
$\lambda_1, \lambda_2, ..., \lambda_p$ and $E_2$ corresponds to the
remaining ones, is $m$-dominated for a $m=m(x)$ large enough. In
fact we have:
\begin{itemize}
\item There exists $N \in \mathbb{N}$ such that $e^{N \left(\lambda_{p+1} - \lambda_{p}\right)} < \frac{1}{2}.$
\item For each $i \in \{1,...,p+1\}$, there is $K^{i}_{x} \in
\mathbb{N}$ such that, for all unit vector $u$ of $U_{i}$ and all
positive integer $n \geq K^{i}_{x}$,
$$\frac{1}{n}\log\|A^{n}(x)u\| \approx \lambda_{i}.$$
\item If $K(x)=\max \, \{N, K^{i}_{x}\}$, then, for all $n \geq
K(x)$, we may conclude that
$$\frac{\|A^{K(x)}(x)(u)\|}{\|A^{K(x)}(x)(v)\|} \, \leq \left(e^{N (\lambda_{p+1} - \lambda_{p})}\right) < \, {1/2}$$
for all $u \in E_{2}$ and $v \in E_{1}$.
\item If $m=\max \, \{K_{f^{j}(x)}: \, j=0, ..., R-1\}$, then
along the orbit of $x$ the Oseledets-Ruelle's splitting is $m$
dominated.
\end{itemize}

\item If $\lambda_{p+1}=-\infty$ and $x$ is periodic by $f$, say of period $R$,
consider, as before, the Oseledets-Ruelle's splitting given by $E_1
\oplus E_2$, where $E_1$ is the subspace associated to the Lyapunov
exponents $\lambda_1, \lambda_2, ..., \lambda_p$ and $E_2$
corresponds to the remaining ones. Then:
\begin{itemize}
\item Since $\lambda_{p} > -\infty$, there exists $\epsilon > 0$ such that
$\lambda_{p} > - \epsilon$.
\item Therefore there exists $N \in \mathbb{N}$ such that $e^{N (- \lambda_{p} - \epsilon)}<
\frac{1}{2}$.
\item For each $i \in \{1,...,p\}$, there is $K^{i}_{x} \in
\mathbb{N}$ such that, for all unit vector $u$ of $U_{i}$ and all
positive integer $n \geq K^{i}_{x}$,
$$\frac{1}{n}\log\|A^{n}(x)u\| \approx \lambda_{i}.$$
\item There exists $N_{1} \in \mathbb{N}$ such that
for all unit vector $u$ of $U_{p+1}$ and all positive integer $n
\geq N_{1}$,
$$\frac{1}{n}\log\|A^{n}(x)u\| \approx -\epsilon.$$
\item If $K(x)=\max \, \{N, N_{1}, K^{i}_{x}\}$, then, for all $n \geq
K(x)$, we may conclude that
$$\frac{\|A^{K(x)}(x)(u)\|}{\|A^{K(x)}(x)(v)\|} \, \leq \left(e^{N (-\lambda_{p} - \epsilon)}\right) < \, {1/2}$$
for all $u \in E_{2}$ and $v \in E_{1}$.
\item If $m=\max \, \{K_{f^{j}(x)}: \, j=0, ..., R-1\}$, then
along the orbit of $x$ the Oseledets-Ruelle's splitting is $m$
dominated.
\end{itemize}

\item If $x$ is non-periodic and the Oseledets-Ruelle's splitting
along the orbit of $x$ is $m$-dominated for some $m$, the proof
ends.

\item Finally if $x$ is non-periodic, belongs to $\Gamma_{p}^{*}(A,m)$ for all $m$
and one of these subsets, say $\Gamma_{p}^{*}(A,m_{0})$, has
positive $\mu$ measure, then the $m_{0}$-domination on $\mathcal{K}$
of the Oseledets-Ruelle's splitting may fail because $x$ is in one
of the corresponding sets $\Gamma_{p,1}^{*}$ or $\Gamma_{p,2}^{*}$.
As seen in the previous corollary, given $\epsilon$, in both cases
there is a cocycle $B \in C_{I}^{0}(X,\mathcal{C}(\mathcal{H}))$
such that $\|A-B\|$ is arbitrarily small but

\begin{itemize}
\item $|LE_{p}(A)-LE_{p}(B)| > \epsilon$, in the case
$\lambda_{p+1}(A) \not= -\infty$
\item $LE_{p}(B)=-\infty$ while $LE_{p}(A)$ is finite, when $\lambda_{p+1}(A)=-\infty$
\end{itemize}

which contradicts the continuity at $A$ of the map $LE_{p}$.
\end{enumerate}
\end{section}

\section*{Acknowledgements} MB was Supported by Funda\c c\~ao para a Ci\^encia e a Tecnologia, SFRH/BPD/20890/2004. MC was Partially supported by Funda\c c\~ao para a Ci\^encia e a Tecnologia through CMUP.

\bigskip
\flushleft
{\bf M\'ario Bessa} \ \  (bessa@fc.up.pt)\\
{\bf Maria Carvalho} \ \  (mpcarval@fc.up.pt)\\
CMUP, Rua do Campo Alegre, 687 \\ 4169-007 Porto \\ Portugal\\

\end{document}